\documentclass[psamsfonts,reqno]{amsart}
\usepackage{amssymb,xy,eucal}
\usepackage[usenames]{color}

\usepackage{graphics}

\usepackage{mathrsfs}

\xyoption{all}

\theoremstyle{plain}

\newtheorem{lemma}{Lemma}[section]
\newtheorem{theorem}[lemma]{Theorem}

\newtheorem{corollary}[lemma]{Corollary}
\newtheorem{claim}{Claim}

\newtheorem*{stat}{\name}
\newcommand{\name}{testing}

\theoremstyle{definition}
\newtheorem{definition}[lemma]{Definition}

\theoremstyle{remark}
\newtheorem{remark}[lemma]{Remark}
\newtheorem{notation}[lemma]{Notation}
\newtheorem{notations}[lemma]{Notations}

\newtheorem*{remark*}{Remark}

\newenvironment{all}[1]{\renewcommand{\name}{#1}\begin{stat}}
                         {\end{stat}}

\newcommand{\qedc}{{\qed}~{\rm Claim~{\theclaim}.}}
\newcommand{\qedsc}{{\qed}~{\rm Claim.}}

\newenvironment{cproof}
{\begin{proof}[Proof of Claim.]}
{\qedc\renewcommand{\qed}{}\end{proof}}

\numberwithin{equation}{section}

\newcommand{\set}[1]{\{#1\}}
\newcommand{\setm}[2]{\set{#1\mid#2}}

\newcommand{\famm}[2]{(#1\mid#2)}

\newcommand{\Sor}{\mathbin{\bigtriangledown}}

\newcommand{\Pow}{\mathfrak{P}}

\newcommand{\dnw}{\mathbin{\downarrow}}
\newcommand{\upw}{\mathbin{\uparrow}}
\newcommand{\Upw}{\mathbin{\Uparrow}}

\newcommand{\scL}{\mathscr{L}}

\newcommand{\cA}{\mathcal{A}}
\newcommand{\cB}{\mathcal{B}}

\newcommand{\cP}{\mathcal{P}}

\newcommand{\cS}{\mathcal{S}}
\newcommand{\cV}{\mathcal{V}}
\newcommand{\cW}{\mathcal{W}}

\newcommand{\into}{\hookrightarrow}
\newcommand{\zero}{\mathbf{0}}

\newcommand{\two}{\mathbf{2}}

\newcommand{\bu}{\boldsymbol{u}}

\newcommand{\bA}{\boldsymbol{A}}
\newcommand{\bB}{\boldsymbol{B}}

\newcommand{\xF}{\mathbf{F}}
\newcommand{\rF}{\mathrm{F}}
\newcommand{\bU}{\boldsymbol{U}}

\newcommand{\id}{\mathrm{id}}
\newcommand{\jz}{$(\vee,0)$}

\newcommand{\jzs}{\jz-semi\-lat\-tice}

\newcommand{\jzh}{\jz-ho\-mo\-mor\-phism}

\newcommand{\res}{\mathbin{\restriction}}

\DeclareMathOperator{\card}{card}

\DeclareMathOperator{\crita}{crit}
\DeclareMathOperator{\critar}{crit_r}
\newcommand{\crit}[2]{\crita({{#1};{#2}})}
\newcommand{\critr}[2]{\critar({{#1};{#2}})}
\newcommand{\module}[1]{|#1|}

\DeclareMathOperator{\Ids}{Id_s}
\DeclareMathOperator{\Fr}{Fr}

\DeclareMathOperator{\At}{At}

\DeclareMathOperator{\dom}{dom}

\DeclareMathOperator{\Con}{Con}
\DeclareMathOperator{\Conc}{Con_c}

\DeclareMathOperator{\Concr}{Con_{c,r}}
\newcommand{\ConV}{\operatorname{\Con^{\cV}}}
\newcommand{\ConcV}{\operatorname{\Con_{\mathrm{c}}^{\cV}}}
\newcommand{\ConcW}{\operatorname{\Con_{\mathrm{c}}^{\cW}}}

\DeclareMathOperator{\Id}{Id}

\newcommand{\tosurj}{\mathbin{\twoheadrightarrow}}
\newcommand{\toinj}{\mathbin{\hookrightarrow}}

\newcommand{\Bool}{\mathbf{Bool}}

\begin{document}

\title[Critical points]{The possible values of critical points between strongly congruence-proper varieties of algebras}

\author[P.~Gillibert]{Pierre Gillibert}
\address{Charles University in Prague, Faculty of Mathematics and Physics, Department of Algebra, Sokolovsk\'a 83, 186 00 Prague, Czech Republic.}
\address{Centro de \'Algebra da Universidade de Lisboa, Av. Prof. Gama Pinto 2, 1649-003 Lisboa, Portugal}
\email{gilliber@karlin.mff.cuni.cz, pgillibert@yahoo.fr}


\thanks{This work was partially supported by the institutional grant MSM 0021620839}

\date{\today}

\begin{abstract}
We denote by~$\Conc A$ the \jzs\ of all finitely generated congruences of an (universal) algebra~$A$, and we define~$\Conc\cV$ as the class of all isomorphic copies of all $\Conc A$, for $A\in\cV$, for any variety~$\cV$ of algebras.

Let~$\cV$ and~$\cW$ be locally finite varieties of algebras such that for each finite algebra $A\in\cV$ there are, up to isomorphism, only finitely many $B\in\cW$ such that $\Conc A\cong\Conc B$, and every such~$B$ is finite. If $\Conc\cV\not\subseteq\Conc\cW$, then there exists a \jzs\ of cardinality~$\aleph_2$ in $(\Conc\cV)-(\Conc\cW)$. Our result extends to quasivarieties of first-order structures, with finitely many relation symbols, and relative congruence lattices.

In particular, if~$\cW$ is a finitely generated variety of algebras, then this occurs in case~$\cW$ omits the tame congruence theory types $\boldsymbol{1}$ and $\boldsymbol{5}$; which, in turn, occurs in case~$\cW$ satisfies a nontrivial congruence identity.

The bound~$\aleph_2$ is sharp.
\end{abstract}

\maketitle

\section{Introduction}\label{S:Intro}
Why do so many representation problems in algebra, enjoying positive solutions in the finite case, have counterexamples of minimal cardinality either~$\aleph_0$, $\aleph_1$, or~$\aleph_2$, and no other cardinality? By a \emph{representation problem}, we mean that we are given \emph{categories}~$\cA$ and~$\cB$ together with a \emph{functor} $\Phi\colon\cA\to\cB$, and we are trying to determine whether an object~$B$ of~$\cB$ is isomorphic to~$\Phi(A)$ for some object~$A$ of~$\cA$. We are also given a mapping from the objects of~$\cB$ to the cardinals, that behaves like the cardinality mapping on sets.

Examples of such representation problems cover various fields of mathematics. Here are a few examples, among many:
\begin{itemize}
\item Every (at most) countable Boolean algebra is generated by a chain (cf. \cite[Theorem~172]{LTF}), but not every Boolean algebra is generated by a chain (cf. \cite[Lemma~179]{LTF}). It is an easy exercise to verify that in fact, every subchain~$C$ of the free Boolean algebra~$F$ on~$\aleph_1$ generators is countable, thus~$F$ cannot be generated by~$C$.

\item Every dimension group with at most~$\aleph_1$ elements is isomorphic to $K_0(R)$ for some (von~Neumann) regular ring~$R$ (cf. \cite{Elli76,GoHa86}), but there is a dimension group with~$\aleph_2$ elements which is not isomorphic to~$K_0(R)$ for any regular ring~$R$ (cf.~\cite{NonMeas}).

\item Every distributive algebraic lattice with at most~$\aleph_1$ compact elements is isomorphic to the congruence lattice of some lattice (cf. \cite{H1,H2,H3}), but not every distributive algebraic lattice is isomorphic to the congruence lattice of some lattice (cf.~\cite{CLP}); the minimal number of compact elements in a counterexample, namely~$\aleph_2$, is obtained in~\cite{R}.
\end{itemize}

In an earlier paper \cite{G1}, we introduced a particular case of the kind of representation problem considered above, concentrated in the notion of \emph{critical point} between two varieties of (universal) algebras. It turned out that this notion often behaves as a paradigm for those kinds of problems. The present paper will be centered on that paradigm, and will offer an explanation, in that context, why for so many representation problems, the minimal size of a counterexample (if it exists at all) lies below~$\aleph_2$. Although initially stated for universal algebras, the method of proof of our main result (Theorem~\ref{T:poscritVar}) carries a potential of generalization to many other contexts, starting with Theorem~\ref{T:poscritQVar}.

Let us be a bit more precise.
For an algebra~$A$ we denote by $\Conc A$ the \jzs\ of all compact (i.e., finitely generated) congruences of~$A$. A \emph{lifting} of a \jzs\ $S$ is an algebra~$A$ such that $\Conc A\cong S$.
For a variety~$\cV$ of algebras we denote by $\Conc\cV$ the class of all \jzs s with a lifting in~$\cV$.

For varieties~$\cV$ and~$\cW$ of algebras, the \emph{critical point} between~$\cV$ and~$\cW$, denoted by $\crit{\cV}{\cW}$, is the smallest cardinality of a member of $(\Conc\cV)-(\Conc\cW)$ if $\Conc\cV\not\subseteq\Conc\cW$, and~$\infty$ otherwise (cf. \cite{G1, CLPSurv}).

The critical point between two varieties can be anything we like. For example, for (possibly infinite) fields~$F$ and~$K$ with $\card F<\card K$, it is easy to verify that
 \[
 \crit{F\text{-vector spaces}}{K\text{-vector spaces}}=\card F+1\,.
 \]
On the other hand, once we fix restrictions on the \emph{similarity types} of our algebras, the situation becomes much more interesting.
All known critical points between varieties of algebras, with a countable similarity type, are either $\le\aleph_2$ or equal to~$\infty$. For example, it is proved in~\cite{RTW} that
 \[
 \crit{\text{lattices}}{\text{groups}}=\crit{\text{lattices}}{\text{rings}}=\aleph_2\,.
 \]
It is also easily seen that $\crit{\text{groups}}{\text{lattices}}=5$. Plo\v s\v cica in \cite{Ploscica09}, using methods introduced by Wehrung in~\cite{CLP} and R\r{u}\v{z}i\v{c}ka in~\cite{R}, finds a majority algebra~$M$ of cardinality~$\aleph_2$ such that $\Conc M$ has no lifting by any lattice. However, every distributive \jzs\ of cardinality $\le\aleph_1$ is liftable by a lattice (cf. \cite{H1,H2,H3}), therefore the critical point between the variety of all majority algebras and the variety of all lattices is~$\aleph_2$.

A strong restriction on the possible values of critical points between finitely generated congruence-distributive varieties of algebras is brought by the following dichotomy result proved in \cite{G1}.

\begin{theorem}\label{T:d1}
Let~$\cV$ be a locally finite variety of algebras and let~$\cW$ be a finitely generated congruence-distributive variety of algebras. If $\Conc\cV\not\subseteq\Conc\cW$, then $\crit{\cV}{\cW}<\aleph_\omega$.
\end{theorem}

Critical points between varieties of lattices have been particularly studied. Tools for proving the countability of certain critical points are given, along with examples, in \cite{Ploscica03b, Ploscica04, G2}. Examples of varieties of lattices with critical point~$\aleph_2$ are given in \cite{Ploscica00, Ploscica03, G2}. In \cite{G1} we give two finitely generated varieties of modular lattices with critical point~$\aleph_1$, solving a problem by T\r{u}ma and Wehrung in~\cite{CLPSurv}. In \cite{G4} we establish the following dichotomy result for varieties of lattices.

\begin{theorem}\label{T:d2}
Let~$\cV$ and~$\cW$ be varieties of lattices such that every simple member of~$\cW$ contains a prime interval. If $\Conc\cV\not\subseteq\Conc\cW$, then $\crit{\cV}{\cW}\le\aleph_2$. Moreover, $\Conc\cV\subseteq\Conc\cW$ if and only if~$\cV$ is contained in either~$\cW$ or its dual.
\end{theorem}

In particular, this solves the finitely generated case of (the correct recasting of) the critical point conjecture for lattices (cf. \cite[Problem~5]{CLPSurv}). However the conjecture can be generalized to arbitrary finitely generated varieties of algebras. This generalization is still open (cf. \cite[Problem~3]{GiWe1}).

A variety~$\cV$ of algebras is \emph{strongly congruence-proper} if every finite \jzs\ has, up to isomorphism, only finitely many liftings in~$\cV$ and every such lifting is finite (cf.~\cite{GiWe1}). As a consequence of \cite[Theorem 10.16]{FM}, a finitely generated variety of congruence-modular algebras, with finite similarity type, is strongly congruence-proper. As observed in \cite[Section~4-10]{GiWe1}, it follows from \cite[Theorem~14.6]{HM} that a finitely generated variety~$\cV$ of algebras, with finite similarity type, that omits tame congruence theory types $\boldsymbol{1}$ and $\boldsymbol{5}$, is strongly congruence-proper. This, in turn, occurs in case~$\cV$ satisfies a nontrivial congruence identity (cf. \cite[Theorem~9.18]{HM}). In particular, this holds for varieties of \emph{lattices}, \emph{groups}, \emph{modules} (over finite rings), \emph{rings}.

Theorem~\ref{T:d1} is generalized in \cite{GiWe1} to strongly congruence-proper varieties of algebras. The aim of the present paper is to improve the bound from~$<\aleph_\omega$ to~$\le\aleph_2$. This solves, in particular, the generalization of the critical point conjecture to the congruence-modular case.

Our main result is the following.

\begin{all}{Theorem~\ref{T:poscritVar}}
Let~$\cV$ and~$\cW$ be locally finite varieties of algebras. Assume that for each finite algebra $A\in\cV$ there are, up to isomorphism, only finitely many $B\in\cW$ such that $\Conc A\cong\Conc B$, and every such~$B$ is finite. Then either $\crit{\cV}{\cW}\le\aleph_2$ or $\Conc\cV\subseteq\Conc\cW$.
\end{all}

In particular, the theorem applies to the case where $\cW$ is strongly congruence-proper (cf. Corollary~\ref{C:poscrit}). The bound~$\aleph_2$ is optimal since there are finitely generated varieties of lattices (hence strongly congruence-proper) with critical point~$\aleph_2$. While Theorem~\ref{T:d2} requires no assumption of either variety~$\cV$ or~$\cW$ be locally finite, we need that assumption for both~$\cV$ and~$\cW$ in the statement of Theorem~\ref{T:poscritVar}. Not every \jzs\ is isomorphic to $\Conc A$ for a locally finite algebra (cf.~\cite{Kear}): this suggests that there is still way to go.

In Section~\ref{S:QVar}, we show how to extend our main result from \emph{varieties of algebras} to \emph{quasivarieties of first-order structures} (for which a most notable example is given by \emph{quasivarieties of graphs}, see~\cite{Gorb}). Theorem~\ref{T:poscritVar} extends, \emph{mutatis mutandis}, to Theorem~\ref{T:poscritQVar}, by assuming finiteness of the set of relation symbols and changing congruences to \emph{relative congruences}.

\section{Basic Concepts}\label{S:Basic}

We denote by $\dom f$ the domain of a function $f$. Given sets $X$ and~$Y$ we denote by $X-Y=\setm{x\in X}{x\not\in Y}$ the set-theoretical difference of $X$ and~$Y$, and by $X^Y$ the set of all maps $f\colon Y\to X$. For a variety~$\cV$ of algebras and a set $X$ we denote by $\Fr_\cV(X)$ the free algebra on~$X$ in~$\cV$.

We denote by $0$ (resp., $1$) the least (resp., largest) element of a poset if it exists. We denote by $\two=\set{0,1}$ the two-element Boolean algebra. We only consider non-empty algebras. We denote by $\zero_A$ the smallest congruence of an algebra~$A$. A \jzh\ $\varphi\colon S\to T$ \emph{separates zero} if $\varphi(x)=0$ implies that $x=0$ for each $x\in S$. Notice that a morphism~$f$ of algebras is an embedding if and only if $\Conc f$ separates zero. In particular if~$f$ and~$g$ are morphisms of algebras and there is a natural equivalence $\Conc f\cong\Conc g$, then~$f$ is an embedding if and only if~$g$ is an embedding.

Let $x<y$ in a poset~$P$, we write $x\prec y$, or, equivalently, $y\succ x$, if there is no $t\in P$ with $x<t<y$. Assume that~$P$ has $0$. An atom of~$P$ is an element $p\in P$ such that $p\succ 0$. We denote by $\At P$ the set of all atoms of $P$.

Let $X$ be a subset of a poset~$P$. We denote
\begin{align*}
P\dnw X &=\setm{p\in P}{(\exists x\in X)(p\le x)}\,,\\
P\upw X &=\setm{p\in P}{(\exists x\in X)(x\le p)}\,,\\
P\Upw X &=\setm{p\in P}{(\forall x\in X)(x\le p)}\,.
\end{align*}
A subset $Q$ of~$P$ is a \emph{lower subset of~$P$} (resp., \emph{upper subset of~$P$}) if $Q=P\dnw Q$ (resp., $Q=P\upw Q$). An upper subset $Q$ of~$P$ is \emph{finitely generated} if $Q = P \upw X$ for some finite subset $X$ of~$P$. An \emph{ideal} of a poset~$P$ is a lower subset $I$ of~$P$ such that for all $x,y\in I$ there is $z\in I$ such that $z\ge x$ and $z\ge y$. An ideal $I$ of~$P$ is \emph{principal} if $I=P\dnw\set{x}$ for some $x\in P$. We say that~$P$ is \emph{lower finite} if every principal ideal of~$P$ is finite.

Let $\theta$ be a congruence of an algebra~$A$. For an element~$a$ of~$A$, we denote by $a/\theta$ the equivalence class of~$a$. For a subset $X$ of~$A$ we set $X/\theta=\setm{x/\theta}{x\in X}$. If $X$ is a subalgebra of~$A$ then $X/\theta$ is a subalgebra of $A/\theta$, moreover $\theta\cap X^2$ is a congruence of $X$. We shall often identify $X/(\theta\cap X^2)$ and $X/\theta$. Given a morphism $f\colon A\to B$ of algebras, the \emph{kernel} of $f$ is $\ker f=\setm{(x,y)\in A^2}{f(x)=f(y)}$. Notice that $\ker f$ is a congruence of~$A$.

Given algebras $A$ and $B$ and $\varphi\colon\Conc A\to\Conc B$ a \jzh, we identify $\varphi$ with its natural extension $\Con A\to\Con B$. That is, given $\theta\in\Con A$, we set $\varphi(\theta)=\bigvee \setm{\varphi(\alpha)}{\alpha\in \Conc A\text{ and }\alpha\le\theta}$.

Cardinals are initial ordinals, in particular a cardinal is identified with a set. We denote by $\Pow(X)$ the set of all subsets of $X$. For a cardinal~$\kappa$, we put
\begin{align*}
[X]^\kappa&=\setm{Y\in\Pow(X)}{\card Y=\kappa}.
\end{align*}

By a \emph{diagram} in a category~$\cS$, we mean a functor from a poset, viewed as a category in the usual way (i.e., there is an arrow from~$p$ to~$q$ if{f} $p\le q$, and then the arrow is unique), to~$\cS$. Hence a $P$-indexed diagram in~$\cS$ is identified with a system $\famm{S_p,\sigma_p^q}{p\leq q\text{ in }P}$ such as $\sigma_p^q\colon S_p\to S_q$ in~$\cS$, $\sigma_p^p=\id_{S_p}$, and $\sigma_p^r=\sigma_q^r\circ\sigma_p^q$, for all $p\le q\le r$ in~$P$.

For an object~$S$ of~$\cS$, the \emph{comma category}, denoted by $\cS\dnw S$, is the category whose objects are the morphisms $f\colon A\to S$ where~$A$ is an object of $\cS$ and the morphisms from $f\colon A\to S$ to $g\colon B\to S$ are the morphisms $h\colon A\to B$ of $\cS$ such that $g\circ h=f$.

Denote epimorphisms by $f\colon A\tosurj B$. Given morphisms $f\colon A\to B$ and $g\colon A\to C$ of algebras, we say that $f$ \emph{factors through} $g$ if there exists $h\colon C\to B$ such that $f=h\circ g$. If $g$ is an epimorphism then the map $h$ is unique.

\section{The condensate construction}\label{S:DefCond}

The proof of the dichotomy result (cf. Theorem~\ref{T:poscritVar}) relies on the \emph{condensate} construction introduced in \cite[Section~3-1]{GiWe1}and the \emph{Armature Lemma} (cf. \cite[Section~3-2]{GiWe1}). In order to ease the understanding of certain crucial parts of our paper, we shall recall the main lines of that construction. The required notions are introduced formally in \cite[Chapter~2]{GiWe1}. {}From Section~\ref{S:liftPID} on, the reader can safely forget most of the notations and definitions introduced in Section~\ref{S:DefCond}, but should keep in mind the crucial Lemma~\ref{L:CC}, which requires the notion of a \emph{norm-covering}~$U$ of a poset $P$. Given a $P$-indexed diagram $\vec S$ in a variety~$\cV$ of algebras, we shall recall the construction of the algebra $\xF(U)\otimes\vec S$ in~$\cV$ and the morphism $\pi_u^U\otimes\vec S\colon \xF(U)\otimes\vec S\to S_{\partial u}$, for each $u\in U$.

The following notation is introduced in \cite[Section~2-1]{GiWe1}.

\begin{notation}
Let $X$ be a subset of a poset~$P$, we denote by $\Sor X$ the set of all minimal elements of $P\Upw X$.
\end{notation}

The following definition is given in \cite[Section~2-1]{GiWe1}.

\begin{definition}
A subset $X$ of a poset~$P$ is \emph{$\Sor$-closed} if $\Sor Y\subseteq X$ for every finite subset $Y$ of $X$. The \emph{$\Sor$-closure} of a subset $X$ of~$P$ is the least $\Sor$-closed subset of~$P$ containing $X$.

We say that~$P$ is \emph{supported} if $P\Upw X$ is a finitely generated upper subset of~$P$ and the $\Sor$-closure of $X$ is finite, for every finite (possibly empty) subset $X$ of~$P$.
\end{definition}

Notice that the definition of a supported poset is equivalent to the one given in \cite[Definition 4.1]{G1}. The \emph{kernels} of a supported poset~$P$ are the finite $\Sor$-closed subsets of~$P$.

The following definition is given in \cite[Definition 4.3]{G1}.

\begin{definition}\label{D:normcov}
A \emph{norm-covering} of a poset~$P$ is a pair $(U,\partial)$ where $U$ is a supported poset and $\partial\colon U\to P$ is an isotone map.

A \emph{sharp ideal} of $(U, \partial)$ is an ideal $\bu$ of $U$ such that $\setm{\partial x}{x\in\bu}$ has a largest element; we denote this element by~$\partial\bu$. We denote by $\Ids(U, \partial)$ the set of all sharp ideals of $(U, \partial)$.
\end{definition}

Notice that this definition of a norm-covering is slightly stronger than \cite[Section~2-6]{GiWe1}. However, this does not affect the definition of an \emph{$\aleph_0$-lifter} (cf. \cite[Section~3-2]{GiWe1}), as in that case we require the norm-covering to be supported.

\begin{remark}
In the context of Definition~\ref{D:normcov}, every principal ideal is sharp. The converse does not hold as a rule. However, in the present paper, we shall only consider norm-coverings for which every sharp ideal is principal, in which case we can identify sharp ideals of $(U, \partial)$ with elements of $U$.
\end{remark}

The following definition comes from \cite[Section~2-2]{GiWe1}.

\begin{definition}
Let~$P$ be a poset. A \emph{$P$-scaled Boolean algebra $\bA$} is a Boolean algebra~$A$, together with a family $\famm{A^{(p)}}{p\in P}$ of ideals of~$A$, such that:
\begin{enumerate}
\item $A=\bigvee \famm{A^{(p)}}{p\in P}$;
\item $A^{(p)}\cap A^{(q)}=\bigvee \famm{A^{(r)}}{\text{$r\ge p,q$ in~$P$}}$, for all $p,q\in P$;
\end{enumerate}
where joins are taken in the lattice of ideals of $A$. As an immediate consequence the assignment $p\mapsto A^{(p)}$ is order-reversing.

A \emph{morphism} $f\colon \bA\to\bB$ of~$P$-scaled Boolean algebras is a morphism $f\colon A\to B$ of Boolean algebras such that $f(A^{(p)})\subseteq B^{(p)}$, for all $p\in P$. We denote by $\Bool_P$ the category of~$P$-scaled Boolean algebras with morphisms of~$P$-scaled Boolean algebras.
\end{definition}

The following definition comes from \cite[Section~2-4]{GiWe1}.

\begin{definition}
A~$P$-scaled Boolean algebra $\bA$ is \emph{compact} if~$A$ is finite and, for each atom~$a$ of~$A$, there is a largest $p\in P$ such that $a\in A^{(p)}$. We set $\module{a}=p$, for this~$p$.
\end{definition}

The finitely presented (in the categorical sense) objects in the category $\Bool_P$ are exactly the compact~$P$-scaled Boolean algebras (cf. \cite[Section~2-4]{GiWe1}). Every~$P$-scaled Boolean algebra is a small directed colimit of compact~$P$-scaled Boolean algebras (cf. \cite[Section~2-4]{GiWe1}).

The following examples of compact $P$-scaled Boolean algebras appear in \cite[Section~2-6]{GiWe1}. They will be used in the proof of Lemma~\ref{L:CC}.

\begin{notation}
Given $p,q$ in a poset $P$ we put
\begin{align*}
\two[p]^{(q)}=
\begin{cases}
\set{0,1},&\text{if $q\le p$}\\
\set{0},&\text{otherwise.}
\end{cases}
\end{align*}
The structure $\two[p]=(\two,\famm{\two[p]^{(q)}}{q\in P})$ is a $P$-scaled Boolean algebra, for each $p\in P$. Moreover, given $p\le q$, the identity map on~$\set{0,1}$ defines a morphism of $P$-scaled Boolean algebras from $\two[p]$ to $\two[q]$; we denote this morphism by $\varepsilon_p^q$.
\end{notation}

We summarize here another family of $P$-scaled Boolean algebras, constructed in \cite[Section 2-6]{GiWe1}. For the sake of simplicity we give the notations only in the cases that we need.

\begin{notations}\label{N:free}
Let $(U,\partial)$ be a norm-covering of a poset $P$. The Boolean algebra $\rF(U)$, defined in \cite[Section~2-6]{GiWe1}, is the Boolean algebra defined by generators $\widetilde u$ (or $\widetilde u^U$ in case $U$ needs to be specified), for $u\in U$, and the relations:
 \begin{align*}
 \widetilde v&\le\widetilde u, && \text{for all $u\le v$ in $U$.}\\
 \widetilde u\wedge \widetilde v &= \bigvee \famm{\widetilde w}
 {w\in\Sor\set{u,v}}, && \text{for all $u,v$ in $U$.}\\
 1 &= \bigvee \famm{\widetilde u}{u\text{ minimal element of }U}.
 \end{align*}
We denote by $\rF(U)^{(p)}$ the ideal of $\rF(U)$ generated by $\setm{\widetilde u}{u\in U\text{ and }p\le \partial u}$. Then $\xF(U)=(\rF(U),\famm{\rF(U)^{(p)}}{p\in P})$ is a $P$-scaled Boolean algebra (cf. \cite[Section~2-6]{GiWe1}).

Given a $\Sor$-closed subset $V$ of $U$, we denote by $f_V^U\colon \rF(V)\to \rF(U)$ the unique morphism of Boolean algebras such that $f_V^U(\widetilde u^V)=\widetilde u^U$ for all $u\in V$. Moreover, $f_V^U$ is a morphism of $P$-scaled Boolean algebras from $\xF(V)$ to $\xF(U)$ (cf. \cite[Section~2-6]{GiWe1}).

Given $u\in U$, we denote by $\pi_u^U\colon\rF(U)\to\two$ the unique morphism of Boolean algebras such that
\begin{align*}
\pi_u^U(\widetilde v)=
\begin{cases}
1, &\text{if $v\le u$}\\
0, &\text{otherwise}
\end{cases},\quad\text{for all $v\in U$}.
\end{align*}
Then $\pi_u^U$ defines a morphism of $P$-scaled Boolean algebras from $\xF(U)$ to $\two[\partial u]$ (cf. \cite[Section~2-6]{GiWe1}).
\end{notations}

The following construction of \emph{condensates} appears in \cite[Section~3-1]{GiWe1}.

\begin{notations}
Let~$\cV$ be a variety of algebras, let $P$ be a poset, and let $\vec S=\famm{S_p,\sigma_p^q}{p\le q\text{ in }P}$ be a $P$-indexed diagram in~$\cV$. Given a compact $P$-scaled Boolean algebra $\bA$, we put
\[
\bA\otimes \vec S=\prod\famm{S_{\module u}}{u\in\At A},
\]
with canonical projections $\delta_{\bA}^u\colon \bA\otimes \vec S\to S_{\module u}$, for all $u\in\At A$.

Let $\varphi\colon\bA\to\bB$ be a morphism of compact $P$-scaled Boolean algebras. Given an atom $v\in B$, we denote by $v^\varphi$ the unique atom $u\in A$ such that $v\le\varphi(u)$. We define $\varphi\otimes \vec S$ as the unique morphism from $\bA\otimes \vec S$ to $\bB\otimes \vec S$ such that $\delta_{\bB}^v\circ(\varphi\otimes \vec S)=\sigma_{\module{v^\varphi}}^{\module{v}}\circ\delta_{\bA}^{v^\varphi}$ for each atom~$v$ of~$B$.

This construction defines a functor $-\otimes\vec S$ from the full subcategory of compact $P$-scaled Boolean algebras of $\Bool_P$ to~$\cV$. It it proved in \cite[Section~1-4]{GiWe1} that this functor can be extended (uniquely up to isomorphism) to the whole category $\Bool_P$ (cf. \cite[Section~3-1]{GiWe1}). We denote by $-\otimes\vec S$ this extension.

An object of the form $\bA\otimes\vec S$, for a $P$-scaled Boolean algebra $\bA$, is called a \emph{condensate} of $\vec S$.
\end{notations}

The construction of $\otimes$ implies the following lemma (cf. \cite[Section~3-1]{GiWe1}).

\begin{lemma}\label{L:epsilonmaps}
Let $\vec S=\famm{S_p,\sigma_p^q}{p\le q\text{ in }P}$ be a $P$-indexed diagram in a variety of algebras~$\cV$. The following equalities hold.
\begin{enumerate}
\item $\two[p]\otimes\vec S=S_p$, for all $p\in P$.
\item $\varepsilon_p^q\otimes\vec S=\sigma_p^q$, for all $p\le q$ in $P$.
\end{enumerate}
\end{lemma}

The following lemma expresses that a condensate of a diagram contains copies of the algebras in the diagram.

\begin{lemma}\label{L:CC}
Let $U$ be a norm-covering of a poset~$P$. Assume that both $U$ and~$P$ have a least element and that $\partial 0=0$. Let $\vec S=\famm{S_p,\sigma_p^q}{p\le q\text{ in }P}$ be a diagram in~$\cV$. There is a morphism $d_0\colon S_0\to \xF(U)\otimes \vec S$ such that $(\pi_u^U\otimes\vec S)\circ d_0 =\sigma_0^{\partial u}$ for each $u\in U$.

Let $u\in U-\set{0}$. There is a morphism $d_u\colon S_0\times S_{\partial u}\to \xF(U)\otimes \vec S$ such that, denoting $t_0\colon S_0\times S_{\partial u}\to S_0$ and $t_1\colon S_0\times S_{\partial u}\to S_{\partial u}$ the canonical projections, the following equality holds
\[
(\pi_v^U\otimes\vec S)\circ d_u=
\begin{cases}
\sigma_0^{\partial v}\circ t_0\,,& \text{if $u\not\le v$}\\
\sigma_{\partial u}^{\partial v}\circ t_1\,,& \text{if $u\le v$}
\end{cases},\quad\text{for each $v\in U$}.
\]
\end{lemma}

\begin{proof}
Notice that $\rF(\set{0})=\set{0,1}$ is the two-element Boolean algebra. Moreover, given $p>0$ in $P$, $\rF(\set{0})^{(p)}=\set{0}$ and $\rF(\set{0})^{(0)}=\set{0,1}$. Hence $\xF(\set{0})=\two[0]$, thus it follows from Lemma~\ref{L:epsilonmaps}(1) that $\xF(\set{0})\otimes \vec S=S_0$. Notice that $\set{0}$ is a $\Sor$-closed subset of $U$. Put $d_0=f_{\set{0}}^U\otimes \vec S$ (cf. Notations~\ref{N:free}). Given $u\in U$ the following equalities hold
\begin{align*}
(\pi_u^U\otimes \vec S)\circ d_0 
& = (\pi_u^U\otimes \vec S)\circ (f_{\set{0}}^U\otimes \vec S)\\
& = (\pi_u^U\circ f_{\set{0}}^U)\otimes \vec S, & & \text{as $-\otimes \vec S$ is a functor}\\
& = \varepsilon_0^{\partial u} \otimes \vec S, &&\text{as $\pi_u^U\circ f_{\set{0}}^U=\varepsilon_0^{\partial u}$}\\
& = \sigma_0^{\partial u}, &&\text{by Lemma~\ref{L:epsilonmaps}(2)}.
\end{align*}
 
Let $u\in U-\set{0}$. Notice that $\rF(\set{0,u})=\set{0,\neg\widetilde u,\widetilde u,1}$, with $\widetilde{0}=1$, is the four-element Boolean algebra. Moreover
\[
\rF(\set{0,u})^{(p)}=
\begin{cases}
\set{0}, & \text{if $p\not\le \partial u$}\\
\set{0,\widetilde u}, & \text{if $0<p\le \partial u$}\\
\set{0,\neg\widetilde u,\widetilde u,1}, &\text{if $p=0$}
\end{cases},\quad\text{for all $p\in P$.}
\]
It follows that $\module{\widetilde u}=\partial u$ and $\module{\neg\widetilde u}=0$, hence $\xF(\set{0,u})\otimes\vec S=S_0\times S_{\partial u}$. Moreover, $\pi_0^{\set{0,u}}$ is the unique morphism from $\xF(\set{0,u})$ to $\two[\partial 0]=\two[0]$ that sends~$\widetilde{u}$ to~$0$, while $\pi_u^{\set{0,u}}$ is the unique morphism from $\xF(\set{0,u})$ to $\two[\partial u]$ that sends~$\widetilde{u}$ to~$1$. It follows that $\pi_0^{\set{0,u}}\otimes\vec S=t_0$ and $\pi_u^{\set{0,u}}\otimes\vec S=t_1$.

Notice that $\set{0,u}$ is a $\Sor$-closed subset of $U$. The following equality holds
 \[
 \pi_v^U \circ f_{\set{0,u}}^U=
 \begin{cases}
 \varepsilon_0^{\partial v}\circ \pi_0^{\set{0,u}}, & \text{if $u\not\le v$}\\
 \varepsilon_{\partial u}^{\partial v}\circ \pi_u^{\set{0,u}}, &
 \text{if $u\le v$}
 \end{cases},\quad\text{for each $v\in U$}.
 \]
Therefore, the map $d_u=f_{\set{0,u}}^U\otimes\vec S$ satisfies the required conditions.
\end{proof}

\section{Lifting poset-indexed diagrams}\label{S:liftPID}

In this section we fix varieties~$\cV$ and~$\cW$ of algebras, an infinite cardinal~$\kappa$, an algebra~$G\in\cW$, and an isomorphism $\xi\colon\Conc \Fr_\cV(\kappa)\to \Conc G$.

Given a finite algebra~$A$, the following lemma expresses that there is a large family of quotients $\Fr_\cV(\kappa)\tosurj A$ such that the corresponding quotients of~$G$ have all the same cardinality.

\begin{lemma}\label{L:TTL}
Let~$A$ be a finite algebra of~$\cV$, fix~$c\in A$. Assume that there is an integer $m$ such that~$\card B\le m$ for every algebra $B\in\cW$ with $\Conc B\cong\Conc A$. For a function~$t$ from a subset of~$\kappa$ to~$A$, we denote by~$\overline t\colon\Fr_{\cV}(\kappa)\to A$ the unique homomorphism extending~$t$ and sending every element of $\kappa-\dom(t)$ to~$c$. For each $X_0\in[\kappa]^{\kappa}$, there are $X\in[X_0]^{\kappa}$ and an integer $n$ such that the following statements hold:
\begin{enumerate}
\item Let $t\colon X\tosurj A$. Then~$\card\bigl(G/\xi(\ker \overline t)\bigr)\le n$.
\item For each~$Y\in[X]^{\kappa}$ there exists $t\colon Y\tosurj A$ such that~$\card\bigl(G/\xi(\ker \overline t)\bigr)= n$.
\end{enumerate}
\end{lemma}

\begin{proof}
Let $t\colon \kappa\tosurj A$. Notice that $\Conc\bigl( G/\xi(\ker \overline t)\bigr) \cong \Conc\bigl(\Fr_\cV(\kappa)/\ker\overline t\bigr)\cong \Conc A$, so
\begin{equation}\label{E:L:PTTL}
\card\bigl(G/\xi(\ker \overline t)\bigr)\le m\,,\quad\text{for each $t\colon \kappa\tosurj A$.}
\end{equation}

Given $X\in[X_0]^{\kappa}$ we set $n_X= \max\setm{\card\bigl(G/\xi(\ker \overline t)\bigr)}{t\colon X\tosurj A}$, note that the maximum exists and is $\le m$ by \eqref{E:L:PTTL}. Moreover 
\begin{equation}\label{E:L:PTTL2}
n_Y\le n_X\,,\quad\text{for all $Y\subseteq X$ in $[X_0]^{\kappa}$.}
\end{equation}

Fix $X\in[X_0]^{\kappa}$ such that $n_X$ is smallest possible. It follows from the definition of $n_X$ that $(1)$ holds for $n=n_X$.

Let~$Y\in[X]^{\kappa}$. From the minimality of $n_X$ we have $n_Y\ge n_X$, the equality $n_Y=n_X$ follows from \eqref{E:L:PTTL2}. Therefore there is $t\colon Y\tosurj A$ such that $\card\bigl(G/\xi(\ker \overline t)\bigr)=n_Y=n_X$. That is $(2)$ holds for $n=n_X$.
\end{proof}

Notice that given $X\in[\kappa]^{\kappa}$ that satisfies the conditions (1) and (2) of Lemma~\ref{L:TTL}, then every $X'\in[X]^{\kappa}$ also satisfies those conditions. Hence an easy induction argument yields the following lemma.

\begin{lemma}\label{L:RTTL}
Let $(A_i)_{i\in I}$ be a finite family of finite algebras of~$\cV$, fix~$c_i\in A_i$ for each $i\in I$. Assume that there is an integer $m$ such that~$\card B\le m$ for each algebra $B\in\cW$ with $\Conc B\cong\Conc A_i$ for some $i\in I$. For a function~$t$ from a subset of~$\kappa$ to~$A_i$, we denote by~$t^{(i)}\colon\Fr_{\cV}(\kappa)\to A_i$ the unique homomorphism extending~$t$ and sending each element of $\kappa-\dom(t)$ to~$c_i$.

Then for each~$X_0 \in[\kappa]^{\kappa}$, there are $X\in[X_0]^{\kappa}$ and a family $(n_i)_{i\in I}$ of integers such that the following statements hold:
\begin{enumerate}
\item Let $t\colon X\tosurj A_i$. Then~$\card\bigl(G/\xi(\ker t^{(i)})\bigr)\le n_i$;
\item For each~$Y\in[X]^{\kappa}$ there exists $t\colon Y\tosurj A_i$ such that~$\card\bigl(G/\xi(\ker t^{(i)})\bigr)= n_i$;
\end{enumerate}
for each $i\in I$.
\end{lemma}

\begin{lemma}\label{R:L:RTTL}
In the context of Lemma~\textup{\ref{L:RTTL}}, if there is an isomorphism $\alpha\colon A_i\to A_j$ such that $\alpha(c_i)=c_j$, then $n_i=n_j$.
\end{lemma}

\begin{proof}
It follows from Lemma~\ref{L:RTTL}(2), applied to $Y=X$, that there exists $t\colon X\tosurj A_i$ such that~$\card\bigl(G/\xi(\ker t^{(i)})\bigr)=n_i$. Put $s=\alpha\circ t$. Notice that $s^{(j)}=\alpha\circ (t^{(i)})$, so $\ker s^{(j)} = \ker t^{(i)}$, thus 
$n_i=\card\bigl(G/\xi(\ker t^{(i)})\bigr) = \card\bigl(G/\xi(\ker s^{(j)})\bigr)\le n_j$. With a similar argument we obtain $n_j\le n_i$.
\end{proof}

The following lemma illustrates that a natural transformation can be factored through a natural equivalence.

\begin{lemma}\label{L:natisodiag}
Let $I$ be a poset, let $\vec G=\famm{G_i,g_i^j}{i\le j\text{ in }I}$ be a diagram of algebras, let $\vec D=\famm{D_i,d_i^j}{i\le j\text{ in }I}$ be a diagram of \jzs s, and let $\vec\chi=\famm{\chi_i}{i\in I}\colon\Conc\circ\vec G\to\vec D$ be a natural transformation. Let $\theta_i$ be a congruence of $G_i$, denote by $p_i\colon G_i\tosurj G_i/\theta_i$ the canonical projection, for each $i\in I$. We assume that each~$\chi_i$ factors through $\Conc p_i$ to an isomorphism. Then $g_i^j$ induces a morphism $\beta_i^j\colon G_i/\theta_i \to G_j/\theta_j$ for all $i\le j$ in $I$. Moreover, $\vec B=\famm{G_i/\theta_i,\beta_i^j}{i\le j\text{ in }I}$ is a diagram of algebras and $\vec\chi$ induces a natural isomorphism from $\Conc\circ\vec G$ to $\vec D$.
\end{lemma}

\begin{proof}
For $i\in I$, we denote by $\tau_i\colon \Conc(G_i/\theta_i)\to D_i$ the isomorphism induced by $\chi_i$, that is, $\tau_i$ is the unique map such that $\chi_i=\tau_i\circ\Conc p_i$. Let $i\le j$ in $I$, let $\alpha\subseteq \theta_i$ be a compact congruence of $B_i$. The following equalities hold
\begin{align*}
(\chi_j\circ\Conc g_i^j)(\alpha)
&=(d_i^j\circ \chi_i)(\alpha), &&\text{as $\vec\chi$ is a natural transformation}\\
&=(d_i^j\circ \tau_i\circ\Conc p_i)(\alpha)\\
&=(d_i^j\circ \tau_i)(\zero_{G_i/\theta_i}),&&\text{as $\alpha\subseteq\theta_i=\ker p_i$}\\
&=0.
\end{align*}
As $\chi_{j}=\tau_{j}\circ\Conc p_{j}$, it follows that $(\Conc p_{j})((\Conc g_i^j)(\alpha))=\zero_{G_j/\theta_j}$, therefore $(\Conc g_i^j)(\alpha)\subseteq\ker p_j=\theta_j$. Let $(x,y)\in\theta_i$. Considering $\alpha=\Theta_{G_i}(x,y)\subseteq\theta_i$, we see that
\[
(g_i^j(x),g_i^j(y))\in (\Conc g_i^j)(\Theta_{G_i}(x,y))\subseteq\theta_j\,.
\]
Therefore $g_i^j$ induces a homomorphism $\beta_i^j\colon G_i/\theta_i\to G_j/\theta_j$. It is easy to check that $\vec B=\famm{G_i/\theta_i,\beta_i^j}{i\le j\text{ in }I}$ is a diagram of algebras. By construction $\beta_i^j\circ p_i=p_j\circ g_i^j$, hence the square (3) of the diagram in Figure~\ref{F:natiso} commutes. As $\vec\chi$ is a natural transformation, the square (1) of the diagram in Figure~\ref{F:natiso} commutes. By definition of $\tau_i$ and $\tau_j$, both triangles (2) and (4) of the diagram in Figure~\ref{F:natiso} commute.

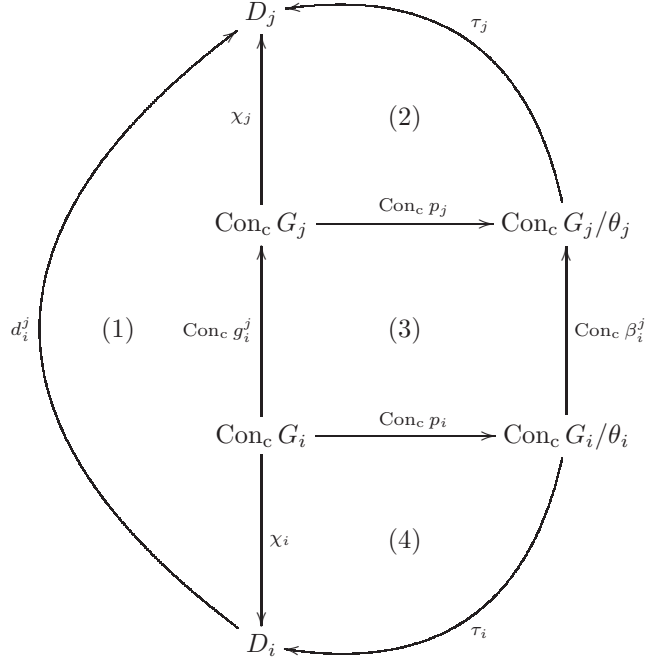
\begin{figure}[htb]
\[
\xymatrix{
&& D_j\\
& & & (2)\\
&& \Conc G_j\ar[uu]^{\chi_j}\ar[rr]^{\Conc p_j} && \Conc G_j/\theta_j\ar@/^-3pc/[uull]_{\tau_j} \\
& (1) & & (3)\\
&& \Conc G_i\ar[uu]^{\Conc g_i^j}\ar[rr]^{\Conc p_i}\ar[dd]^{\chi_i} && \Conc G_i/\theta_i\ar@/^3pc/[ddll]^{\tau_i}\ar[uu]_{\Conc \beta_i^j}\\
& & & (4) \\
&& D_i \ar@/^7pc/[uuuuuu]^{d_i^j}
}
\]
\caption{The family $\vec\tau$ is a natural isomorphism}\label{F:natiso}
\end{figure}

Hence we have proved that the diagram in Figure~\ref{F:natiso} commutes. It follows that $\tau_j\circ(\Conc \beta_i^j)=d_i^j\circ\tau_i$, moreover $\tau_i$ is an isomorphism, for all $i\le j$ in $I$. Therefore, $\vec\tau=\famm{\tau_i}{i\in I}$ defines a natural equivalence from $\Conc\circ\vec B$ to $\vec D$.
\end{proof}

An \emph{$\aleph_0$-lifter} of a poset $P$ is a norm-covering $(U,\partial)$ of $P$, endowed with a set of sharp ideals $\bU\subseteq\Id_s(U,\partial)$ and for which the map~$\partial\colon U\to P$ has a collection of right inverses satisfying certain infinite combinatorial properties (cf. \cite[Section~3-2]{GiWe1}). This property implies, in particular, that~$\partial$ is \emph{surjective}.

\begin{lemma}\label{L:MainTech}
Let~$I$ be a finite lattice, let $\vec A=\famm{A_i,\alpha_i^j}{i\le j\text{ in }I}$ be a diagram in~$\cV$. Let $(U,\bU)$ be an~$\aleph_0$-lifter of~$I$ such that $\card U\le\kappa$. We assume that the following statements hold.
\begin{enumerate}
\item Every element of $\bU$ is a principal ideal.
\item The poset $U$ is lower finite and has a least element.
\item $\partial u\succ 0$ implies that $u\succ 0$ for each $u\in U$.
\item There is an integer $m$ such that~$\card B\le m$ for each algebra $B\in\cW$ with $\Conc B\cong\Conc A_i$ for some $i\in I$.
\item The algebra $A_0$ is generated by one element.
\item The algebra $G$ is locally finite.
\end{enumerate}

Then there exists a lifting $\vec B=\famm{B_i,\beta_i^j}{i\le j\text{ in }I}$ of $\Conc\circ\vec A$ in~$\cW$, such that for all $0\prec i\leq j$ in~$I$, if $\alpha_i^j$ is an isomorphism, then $\beta_i^j$ is an isomorphism.
\end{lemma}

\begin{proof}
{}From the surjectivity of~$\partial$ it follows that~$\partial0=0$.

Denote by~$c_0$ a generator of $A_0$ and put~$c_i=\alpha_0^i(c_0)$ for all $i\in I$. For a function~$t$ from a subset of~$\kappa$ to~$A_i$, we denote by~$t^{(i)}\colon\Fr_{\cV}(\kappa)\to A_i$ the unique homomorphism extending~$t$ and sending each element of $\kappa-\dom(t)$ to~$c_i$. By Lemma~\ref{L:RTTL}, there are $X\in[\kappa]^{\kappa}$ and a family $(n_i)_{i\in I}$ of integers such that the following statements hold:
\begin{enumerate}
\item[$(7)$] Let $i\in I$ and let $t\colon X\tosurj A_i$. Then~$\card\bigl(G/\xi(\ker t^{(i)})\bigr)\le n_i$.
\item[$(8)$] Let~$Y\in[X]^{\kappa}$, let $i\in I$. There exists $t\colon Y\tosurj A_i$, such that~$\card\bigl(G/\xi(\ker t^{(i)})\bigr)= n_i$.
\end{enumerate}
 Moreover, we can assume that~$\kappa-X$ is not empty.

\begin{claim}\label{Cl:1}
There exists a morphism $f\colon \Fr_\cV(\kappa)\to \xF(U)\otimes \vec A$ such that the following statements hold:
\begin{enumerate}
\item[$(9)$] The morphism $(\pi_u^U\otimes \vec A)\circ f$ is surjective, for each $u\in U$.
\item[$(10)$] Let $0\prec u\le v$ in $U$. If $\alpha_{\partial u}^{\partial v}$ is an isomorphism, then
\[\card\bigl(G/\xi(\ker ((\pi_u^U\otimes \vec A)\circ f))\bigr) \ge \card\bigl(G/\xi(\ker ((\pi_v^U\otimes \vec A)\circ f))\bigr).\]
\end{enumerate}
\end{claim}

\begin{cproof}
Set $k_u=\pi_u^U\otimes\vec A$, for each $u\in U$. Put $U^*=U-\set{0}$. Fix morphisms $d_0\colon A_0\to \xF(U)\otimes\vec A$ and $d_u\colon A_0\times A_{\partial u}\to \xF(U)\otimes\vec A$, for $u\in U^*$, as in Lemma~\ref{L:CC} (with $\vec S$ replaced by~$\vec A$). In particular
\begin{align}
k_u\circ d_0 (a)&=\alpha_0^{\partial u}(a),&&\text{for all }u\in U\text{ and all }a\in A_0\,.
\label{E:C1E1}\\
k_v\circ d_u (a,b)&=\alpha_0^{\partial v}(a),&&\text{for all }u\not\le v\text{ in }U\,,\ a\in A_0\,,\text{ and }b\in A_{\partial u}\,.\label{E:C1E2}\\
k_v\circ d_u (a,b)&=\alpha_{\partial u}^{\partial v}(b),&&\text{for all }0<u\le v\text{ in }U\,,\ a\in A_0\,,\ \text{ and }b\in A_{\partial u}\,.\label{E:C1E3}
\end{align}
As~$\card U^*\le\kappa=\card X$, there is a partition $(X_u)_{u\in U^*}$ of~$X$ such that~$\card X_u=\nobreak\kappa$ for each $u\in U$. Put~$X_0=\kappa-X$, so $(X_u)_{u\in U}$ is a partition of~$\kappa$. Denote by $f_0\colon X_0\to \xF(U)\otimes \vec A$, $x\mapsto d_0(c_0)$ the constant map. It follows from~\eqref{E:C1E1} that
\[
k_v(f_0(x))=k_v(d_0(c_0))=\alpha_0^{\partial v}(c_0)=c_{\partial v},\quad\text{for all }x\in X_0\text{ and all }v\in U\,.
\]
Thus the following equality holds.
\begin{equation}\label{E:C1E4}
k_v\circ f_0=c_{\partial v},\text{ the constant map, for all $v\in U$}.
\end{equation}
Let $u\in U^*$. As~$X_u\subseteq X$ and~$\card X_u=\kappa$, it follows from (8) that there exists $t_u\colon X_u\tosurj A_{\partial u}$ such that the following equality holds
 \begin{equation}\label{E:C1E5}
 \card\bigl(G/\xi(\ker t_u^{(\partial u)})\bigr)=n_{\partial u}.
 \end{equation}
Put $f_u\colon X_u\to \xF(U)\otimes \vec A$, $x\mapsto d_u(c_0,t_u(x))$. {}From~\eqref{E:C1E3} we obtain $k_u(f_u(x))=\alpha_{\partial u}^{\partial u}(t_u(x))=t_u(x)$, for all $x\in X_u$. Hence
 \begin{equation}\label{E:C1E6}
 k_u\circ f_u = t_u,\quad \text{for each $u\in U^*$\,.}
 \end{equation}
Similarly, \eqref{E:C1E2} implies that
 \begin{equation}\label{E:C1E7}
 k_v\circ f_u = c_{\partial v},\text{ the constant map, for each $u\not\le v$ in $U$\,.}
 \end{equation}
The family $(X_u)_{u\in U}$ is a partition of~$\kappa$, so there is a (unique) morphism of algebras $f\colon \Fr_\cV(\kappa)\to \xF(U)\otimes \vec A$ that extends $f_u$ for each $u\in U$. Let $u\in U^*$. As $t_u$ is surjective it follows from~\eqref{E:C1E6} that $k_u\circ f$ is surjective. As~$X_0$ is not empty, we see from~\eqref{E:C1E4} that the image of $k_0\circ f$ contains (as an element)~$c_0$, which is a generator of $A_0$, so $k_0\circ f$ is surjective. Therefore, $f$ satisfies (9).

Let $u\succ 0$ in $U$. Let $x\in\kappa$, let $v$ in $U$ be such that $x\in X_v$. If $v\not\le u$, then it follows from \eqref{E:C1E7} that $k_u\circ f(x)=c_{\partial u}= t_u^{(\partial u)}(x)$. If $v<u$, then $v=0$, hence \eqref{E:C1E4} implies that $k_u\circ f(x)=c_{\partial u}=t_u^{(\partial u)}(x)$. If $v=u$, then from \eqref{E:C1E6} we obtain that $k_u\circ f(x)=k_u\circ f_v(x)=t_u(x)=t_u^{(\partial u)}(x)$. Hence $k_u\circ f=t_u^{(\partial u)}$, therefore~\eqref{E:C1E5} implies that
 \begin{equation}\label{E:C1E8}
 \card\bigl(G/\xi(\ker(k_u\circ f))\bigr)=n_{\partial u},
 \quad\text{for each $u\succ 0$ in $U$}.
 \end{equation}
Let $0\prec u\le v$ in $U$ such that $\alpha_{\partial u}^{\partial v}$ is an isomorphism and set $t=k_v\circ f\res_X$. 
It follows from~\eqref{E:C1E6} that the range of~$t$ contains the range of~$t_v$, which is equal to~$A_{\partial v}$; hence, $k_v\circ f\res_X$ is a surjective map from~$X$ onto~$A_{\partial v}$. For $x\in X_0=\kappa-X$ the following equalities hold: $k_v(f(x))=k_v(f_0(x))=c_{\partial v}$ by \eqref{E:C1E4}. It follows that $k_v\circ f=t^{(\partial v)}$. Therefore, by~(7),
 \[
 \card\bigl(G/\xi(\ker(k_v\circ f))\bigr)=\card\bigl(G/\xi(\ker t^{(\partial v)})\bigr)
 \le n_{\partial v}\,.
 \]
However, $n_{\partial v}=n_{\partial u}$ (cf. Lemma~\ref{R:L:RTTL}). Thus, by~\eqref{E:C1E8}, we obtain that\linebreak $\card\bigl(G/\xi(\ker(k_u\circ f))\bigr)=n_{\partial u}\ge \card\bigl(G/\xi(\ker(k_v\circ f))\bigr)$.
\end{cproof}

We fix a morphism $f\colon \Fr_\cV(\kappa)\to \xF(U)\otimes \vec A$ as in Claim~1. Put~$\chi=(\Conc f)\circ\xi^{-1}$. Put $\psi_u = (\pi_u^U\otimes\vec A)\circ f$. By $(9)$, the morphism $\psi_u$ is surjective. Put $\theta_u=\ker \psi_u$ and $\theta_u'=\xi(\theta_u)$. Denote by $h_u\colon \Fr_\cV(\kappa)/\theta_u \to A_{\partial u}$ the morphism induced by $\psi_u$. As $\psi_u$ is surjective, $h_u$ is an isomorphism. In particular,
 \[
 \Conc(G/\theta'_u)\cong\Conc(\Fr_\cV(\kappa)/\theta_u)\cong\Conc A_{\partial u}\,,
 \]
thus, by Assumption~(4), $G/\theta'_u$ is finite.

As~$G$ is locally finite, there is a finite subalgebra~$G_u$ of~$G$ such that~$G_u/\theta_u' = G/\theta_u'$. As~$U$ is lower finite, changing~$G_u$ to the subalgebra of~$G$ generated by $\bigcup_{v\le u}G_v$ makes it possible to assume that~$G_u$ is contained in~$G_v$ for all $u\le v$ in~$U$.

Denote by $\cS$ the category of all \jzs s with \jzh s and put $S=\Conc G$. For all $u\le v$ in~$U$, set $S_u=\Conc G_u$, denote by $g_u\colon G_u\into G$ and $g_u^v\colon G_u\into G_v$ the inclusion maps, and set $\varphi_u^v=\Conc g_u^v$ and $\varphi_u=\Conc g_u$. This defines a diagram $\famm{(S_u,\varphi_u),\varphi_u^v}{u\le v\text{ in }U}$ in~$\cS\dnw S$. We set
 \[
 \rho_u=\Conc(\pi_u^U\otimes \vec A)\circ\chi=(\Conc \psi_u)\circ\xi^{-1}\,,
 \qquad\text{for each }u\in U\,.
 \]
Denote by $p_u\colon G\tosurj G/\theta_u'$ and $p'_u\colon G_u \tosurj G_u/\theta_u'$ the canonical projections, for each $u\in U$. 

\begin{claim}\label{Cl:2}
The map $\rho_u\circ\varphi_u$ factors, through $\Conc p_u'$, to an isomorphism.
\end{claim}

\begin{cproof}
As $\rho_u=(\Conc \psi_u)\circ\xi^{-1}$, the square (3) of the diagram in Figure~\ref{F:rhophiII} commutes.

Denote by $g'_u\colon G_u/\theta_u'\to G/\theta_u'$ the morphism induced by the inclusion map~$g_u$, so the square~(1) of the diagram in Figure~\ref{F:rhophiII} commutes. It follows from the choice of~$G_u$ that $g'_u$ is an isomorphism.

Denote by $\xi'_u\colon \Conc(G/\theta_u') \to\Conc\bigl( \Fr_{\cV}(\kappa) /\theta_u\bigr)$ the \jzh\ induced by $\xi^{-1}$. Then the square~(2) of the diagram in Figure~\ref{F:rhophiII} commutes. As $\theta_u'=\xi(\theta_u)$ and $\xi^{-1}$ is an isomorphism, $\xi_u'$ is an isomorphism.

Denote by $q_u\colon\Fr_{\cV}(\kappa)\tosurj\Fr_{\cV}(\kappa)/\theta_u$ the canonical projection, for each $u\in U$. As $\psi_u=h_u\circ q_u$, the triangle~(4) of the diagram in Figure~\ref{F:rhophiII} commutes.

\begin{figure}[htb]
\[
\xymatrix{
& & \Conc(G/\theta_u') \ar[rr]^{\xi_u'} & &
\Conc\bigl(\Fr_\cV(\kappa)/\theta_u\bigr)\ar@/^5pc/[dddd]^{\Conc h_u}\\
& & & (2) ~\\
\Conc(G_u/\theta_u') \ar[uurr]^{\Conc g_u'} & (1) &\Conc G \ar[uu]_{\Conc p_u}\ar[rr]^{\xi^{-1}} & &\Conc \Fr_\cV(\kappa) \ar[dd]_{\Conc \psi_u} \ar[uu]^{\Conc q_u} & \kern-4pc (4)\\
& & & (3) ~\\
& & \Conc G_u \ar[uull]^{\Conc p'_u} \ar[uu]^{\varphi_u}_{=\Conc g_u} \ar[rr]_{\rho_u\circ\varphi_u} && \Conc A_{\partial u}
}
\]
\caption{The map $\rho_u\circ\varphi_u$ factors through $\Conc p'_u$}\label{F:rhophiII}
\end{figure}
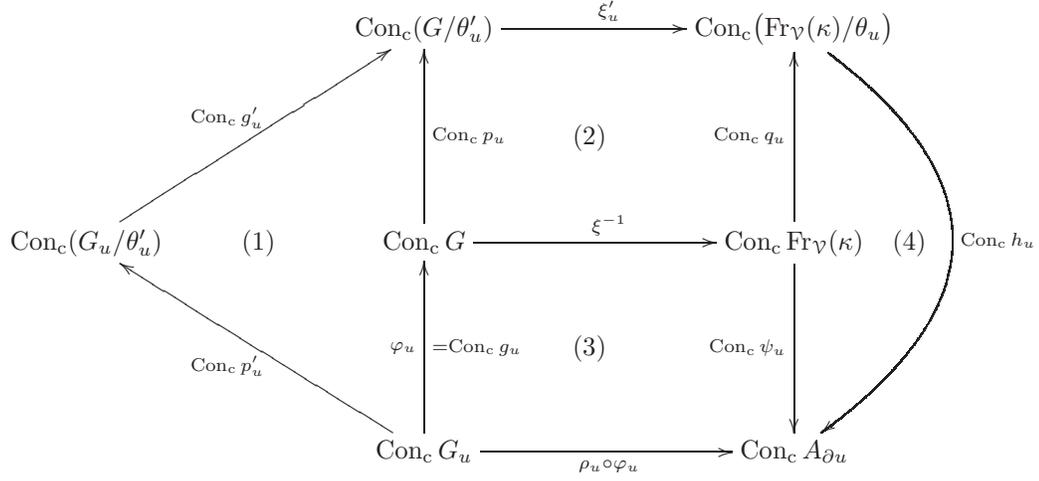

Therefore, the diagram in Figure~\ref{F:rhophiII} commutes, hence
 \[
 \rho_u\circ\varphi_u=(\Conc h_u)\circ \xi_u'\circ(\Conc g'_u)\circ (\Conc p_u')\,.
 \]
As $(\Conc h_u)\circ \xi_u'\circ(\Conc g'_u)$ is an isomorphism, the conclusion follows.
\end{cproof}

As~$G_u$ is finite, $S_u=\Conc G_u$ is finite. By applying the Armature Lemma (cf. \cite[Section~3-2]{GiWe1}) to the functor~$\Conc$ on~$\cV$, we obtain that there is an isotone section $\sigma\colon I\toinj U$ such that the family $\famm{\rho_{\sigma(i)}\circ\varphi_{\sigma(i)}}{i\in I}$ is a natural transformation from $\famm{S_{\sigma(i)},\varphi_{\sigma(i)}^{\sigma(j)}}{i\le j\text{ in }I}$ to $\Conc\circ\vec A$.

It follows from Claim~2 that the map $\rho_{\sigma(i)}\circ\varphi_{\sigma(i)}$ induces an isomorphism $\tau_i\colon \Conc\bigl(G_{\sigma(i)}/\theta_{\sigma(i)}'\bigr)\to\Conc A_i$ for each $i\in I$. Put $B_i=G_{\sigma(i)}/\theta_{\sigma(i)}'$. It follows from Lemma~\ref{L:natisodiag} that $g_{\sigma(i)}^{\sigma(j)}$ induces a morphism $\beta_i^j\colon B_i \to B_j$ for all $i\le j$ in~$I$. This defines a diagram $\vec B=\famm{B_i,\beta_i^j}{i\le j\text{ in }I}$, and $\vec\tau$ is a natural equivalence from $\Conc\circ\vec B$ to $\Conc\circ\vec A$.

Let $0\prec i\le j$ such that $\alpha_i^j$ is an isomorphism. In particular, $\alpha_i^j$ is an embedding, thus $\beta_i^j$ is also an embedding. As $\partial\sigma(i)=i\succ 0$, it follows from Assumption~$(3)$ that $\sigma(i)\succ 0$. Thus, from Assumption~$(10)$ (cf. Claim~\ref{Cl:1}) we obtain that~$\card\bigl(G/\theta_{\sigma(i)}'\bigr)\ge \card\bigl(G/\theta_{\sigma(j)}'\bigr)$. However $B_i=G_{\sigma(i)}/\theta_{\sigma(i)}'=G/\theta_{\sigma(i)}'$, similarly $B_j=G/\theta_{\sigma(j)}'$, so it follows that $\card B_i\ge\card B_j$. However, $\beta_i^j\colon B_i\to B_j$ is an embedding, therefore $\beta_i^j$ is an isomorphism.
\end{proof}

The following lemma is proved in \cite[Lemma~4.11]{G2}. We remind the reader that an \emph{$\aleph_0$-compatible} norm-covering, as defined in~\cite[Definition~4.4]{G2}, is nothing else as an $\aleph_0$-lifter.

\begin{lemma}\label{L:I-compatible-norm-covering}
Let~$X$ be a finite set and set~$\cP=\setm{P\subseteq X}{\card P\le 2\text{ or }P=X}$.
Define~$U$ as the set of all functions from a subset of~$X$ to~$\aleph_2$, partially ordered by inclusion. Let
\begin{align*}
\partial\colon U &\to \cP\\
u &\mapsto \partial u=
\begin{cases}
\dom u &\text{if~$\card (\dom u)\le 2$},\\
X &\text{otherwise}.
\end{cases}
\end{align*}
Denote by $\bU$ the set of all principal ideals of $U$. Then $(U,\bU)$ is an~$\aleph_0$-lifter of~$\cP$. Moreover, $\card U=\aleph_2$.
\end{lemma}

\begin{remark}\label{R:normcov}
In the context of Lemma~\ref{L:I-compatible-norm-covering}, the poset $U$ is lower finite and has a smallest element. Moreover, $\partial u\succ 0$ implies that $\dom u=\set{x}$ for some $x\in X$, hence $u\succ 0$. Hence the truncated Boolean algebra~$\cP$ has an $\aleph_0$-lifter that satisfies the conditions $(1)$, $(2)$, and $(3)$ of Lemma~\ref{L:MainTech}.
\end{remark}

The following lemma expresses that a diagram of \jzs s with a lifting in~$\cV$ has a lifting in~$\cW$.

\begin{corollary}\label{C:LiftAllDiag}
Assume that~$\kappa=\aleph_2$ and that~$G$ is locally finite. Let~$\vec A$ be a diagram of finite algebras and embeddings in~$\cV$, indexed by a finite \jzs~$I$. Assume that for each algebra~$A$ of the diagram~$\vec A$, there is a finite bound on the cardinality of liftings of $\Conc A$ in~$\cW$. Then the diagram $\Conc\circ\vec A$ is liftable in~$\cW$.
\end{corollary}

\begin{proof}
Write $\vec A=\famm{A_i,\alpha_i^j}{i\le j\text{ in }I}$ and put~$\cP=\setm{P\subseteq I}{\card P\le 2\text{ or }P=I}$. Let us choose~$c\in A_0$, denote by $A'_\emptyset$ the subalgebra of $A_0$ generated by $c$, and put $A'_P=A_{\textstyle{\vee} P}$ for each nonempty $P\in\cP$. Put $f_P^Q=\alpha_{\textstyle{\vee} P}^{\textstyle{\vee} Q}\colon A_P'\to A_Q'$ for all $P\subseteq Q$ in~$\cP$. This defines a diagram $\vec A'=\famm{A_P',f_P^Q}{P\subseteq Q\text{ in }\cP}$.

It follows from Lemma~\ref{L:I-compatible-norm-covering} and Remark~{R:normcov} that there is an~$\aleph_0$-lifter $(U,\bU)$ of~$\cP$, such that $\card U=\aleph_2$ and the following statements hold:

\begin{enumerate}
\item Every element of $\bU$ is a principal ideal.
\item The poset $U$ is lower finite and has a smallest element.
\item $\partial u\succ 0$ implies that $u\succ 0$ for each $u\in U$.
\end{enumerate}
The assumptions of Corollary~\ref{C:LiftAllDiag} imply that:
\begin{enumerate}
\item[$(4)$] There is an integer $m$ such that~$\card B\le m$ for each algebra $B\in\cW$ with $\Conc B\cong\Conc A_i$ for some $i\in I$.
\end{enumerate}

Moreover, by construction, $A'_\emptyset$ is generated by one element and we have assumed that $G$ is locally finite. It follows from Lemma~\ref{L:MainTech} that there exists a lifting $\vec B=\famm{B_P,\beta_P^Q}{P\subseteq Q\text{ in }\cP}$ of $\Conc\circ\vec A'$ in~$\cW$, such that for all $\emptyset\prec P\subseteq Q$ in~$\cP$, if $f_P^Q$ is an isomorphism, then $\beta_P^Q$ is an isomorphism.

Put $C_i=B_{\set{i}}$, for all $i\in I$. Let $i\le j$ in~$I$. Notice that $f_{\set{j}}^{\set{i,j}}=\alpha_j^j=\id_{A_j}$ is an isomorphism, hence $\beta_{\set{j}}^{\set{i,j}}$ is an isomorphism. The map $g_i^j=\left(\beta_{\set{j}}^{\set{i,j}}\right)^{-1}\circ\beta_{\set{i}}^{\set{i,j}}$ is a morphism from~$C_i$ to~$C_j$.

The following equalities hold
\[\beta_{\set{j}}^I\circ g_i^j=\beta_{\set{i,j}}^I\circ \beta_{\set{j}}^{\set{i,j}}\circ \left(\beta_{\set{j}}^{\set{i,j}}\right)^{-1}\circ\beta_{\set{i}}^{\set{i,j}} = \beta_{\set{i,j}}^I\circ\beta_{\set{i}}^{\set{i,j}} =\beta_{\set{i}}^{I}\,.\]
Hence
\begin{equation}\label{E:C:LiftAllDiag}
\beta_{\set{j}}^I\circ g_i^j = \beta_{\set{i}}^I,\quad\text{for all $i\le j$ in~$I$\,.}
\end{equation}
Let $i\le j\le k$, it follows from \eqref{E:C:LiftAllDiag} that
\[
\beta_{\set{k}}^I\circ g_j^k\circ g_i^j=\beta_{\set{j}}^I\circ g_i^j=\beta_{\set{i}}^I=\beta_{\set{k}}^I\circ g_i^k\,.
\]
Moreover, $f_{\set k}^I=\alpha_k^1$ is an embedding, hence $\beta_{\set{k}}^I$ is an embedding, thus $g_j^k\circ g_i^j=g_i^k$. Therefore, $\famm{C_i,g_i^j}{i\le j\text{ in }I}$ is a diagram of algebras in~$\cW$.

Let $\vec\tau=(\tau_P)_{P\in\cP}\colon\Conc\circ\vec A'\to\Conc\circ\vec B$ be a natural equivalence. Let $i\le j$ in~$I$. As $\vec \tau$ is a natural equivalence, $\bigl(\Conc\beta_{\set{j}}^{\set{i,j}}\bigr)\circ\tau_{\set j} = \tau_{\set{i,j}}\circ\bigl(\Conc f_{\set{j}}^{\set{i,j}}\bigr)$. However $f_{\set{j}}^{\set{i,j}}=\alpha_j^j=\id_{A_j}$, hence $\Conc\beta_{\set{j}}^{\set{i,j}}\circ\tau_{\set j}=\tau_{\set{i,j}}$, thus the following equality holds
\begin{equation}\label{E:condtau}
\tau_{\set j}=(\Conc\beta_{\set{j}}^{\set{i,j}})^{-1}\circ\tau_{\set{i,j}}\,.
\end{equation}
Therefore, we obtain
\begin{align*}
(\Conc g_i^j)\circ \tau_{\set i} &= 
\left(\Conc\beta_{\set{j}}^{\set{i,j}}\right)^{-1}\circ(\Conc\beta_{\set{i}}^{\set{i,j}})\circ \tau_{\set i}, &&\text{by the definition of $g_i^j$}\\
&=\left(\Conc\beta_{\set{j}}^{\set{i,j}}\right)^{-1}\circ\tau_{\set{i,j}}\circ(\Conc f_{\set{i}}^{\set{i,j}}), &&\text{as $\vec\tau$ is a natural equivalence}\\
&=\tau_{\set{j}} \circ(\Conc f_{\set i}^{\set{i,j}} ), &&\text{by \eqref{E:condtau}}\\
&=\tau_{\set{j}} \circ(\Conc \alpha_i^j),&&\text{as $f_{\set i}^{\set{i,j}}=\alpha_i^j$.}
\end{align*}
Hence $(\tau_{\set{i}})_{i\in I}$ is a natural equivalence from $\Conc\circ\vec A$ to $\Conc\circ\vec C$.
\end{proof}

\section{Critical points}

We can now prove the main result of this paper.

\begin{theorem}\label{T:poscritVar}
Let~$\cV$ and~$\cW$ be locally finite varieties of algebras. Assume that for each finite algebra $A\in\cV$ there are, up to isomorphism, only finitely many $B\in\cW$ such that $\Conc A\cong\Conc B$, and every such~$B$ is finite. Then either $\crit{\cV}{\cW}\le\aleph_2$ or $\Conc\cV\subseteq\Conc\cW$.
\end{theorem}

\begin{proof}
Assume that $\crit{\cV}{\cW}>\aleph_2$. The algebra $\Fr_\cV(\aleph_2)$ is locally finite, so $\card \Fr_\cV(\aleph_2)\le\aleph_2$, hence $\card\Conc \Fr_\cV(\aleph_2)\le\aleph_2$. There are~$G\in\cW$ and an isomorphism $\xi\colon\Conc \Fr_\cV(\aleph_2)\to\Conc G$.

The remaining of the proof is similar to the Dichotomy Theorem of \cite[Section~4-9]{GiWe1}, using Corollary~\ref{C:LiftAllDiag}.

Let $A\in\cV$, let $a\in A$. Denote by $P$ the set of all finite subalgebras of~$A$ containing~$a$. Set $A_p=p$, denote by $\alpha_p^q\colon A_p\to A_q$, and by $\alpha_p\colon A_p\to A$ the inclusion maps, for all $p\le q$ in $P$. Put $\vec A=\famm{A_p,\alpha_p^q}{p\le q\text{ in }P}$. As~$\cV$ is locally finite, the poset~$P$ is a \jzs\ and $\famm{A,\alpha_p}{p\in P}$ is a colimit cocone of~$\vec A$.

Let $I$ be a finite \jz-subsemilattice of $P$, by Corollary~\ref{C:LiftAllDiag} the diagram $\Conc\circ\vec A\res_I$ has a lifting in~$\cW$, it follows from the Compactness Lemma of \cite[Section~4-9]{GiWe1} that $\Conc\circ\vec A$ has a lifting~$\vec B$ in~$\cW$. Fix $B\in\cW$ a colimit of~$\vec B$. The following isomorphisms hold
\begin{align*}
\Conc A&\cong\Conc(\varinjlim\vec A), &&\text{as~$A$ is a colimit of~$\vec A$}\\
&\cong\varinjlim(\Conc\circ\vec A), &&\text{as $\Conc$ preserves directed colimits}\\
&\cong\varinjlim(\Conc\circ\vec B), &&\text{as $\Conc\circ\vec A$ and $\Conc\circ\vec B$ are naturally isomorphic}\\
&\cong\Conc(\varinjlim\vec B), &&\text{as $\Conc$ preserves directed colimits}\\
&\cong\Conc B, &&\text{as $B$ is a colimit of~$\vec B$}.
\end{align*}
Hence $\Conc A$ has a lifting in~$\cW$ for each $A\in\cV$, that is, $\Conc\cV\subseteq\Conc\cW$.
\end{proof}

If~$\cW$ is strongly congruence-proper, then the condition of Theorem~\ref{T:poscritVar} is satisfied, we deduce the following result.

\begin{corollary}\label{C:poscrit}
Let~$\cV$ and~$\cW$ be locally finite varieties of algebras. If~$\cW$ is strongly congruence-proper, then either $\crit{\cV}{\cW}\le\aleph_2$ or $\Conc\cV\subseteq\Conc\cW$.
\end{corollary}

By using the results of~\cite{HM}, we observed, in \cite[Section~4-10]{GiWe1}, that a finitely generated variety of algebras that satisfies a nontrivial congruence identity is strongly congruence-proper. Therefore, the following result is a consequence of Corollary~\ref{C:poscrit}.

\begin{corollary}
Let $\cV$ be a locally finite variety of algebras. Let $\cW$ be a finitely generated variety of algebras, with finite similarity type. If~$\cW$ satisfies a nontrivial congruence identity, then either $\crit{\cV}{\cW}\le\aleph_2$ or $\Conc\cV\subseteq\Conc\cW$.
\end{corollary}

\section{Extension to quasivarieties of first-order structures}\label{S:QVar}

We conclude the paper with a word on \emph{quasivarieties of first-order structures}, as considered in~\cite{Gorb}. We briefly recall the basic definitions. A class of first-order structures on a first-order language~$\scL$ is a \emph{quasivariety} if it is closed under substructures, direct products, and directed colimits (within the class of all models for~$\scL$). This notion is quite robust and has many equivalent forms, see~\cite{Gorb} for details. In \cite[Section~4-1]{GiWe1}, we define a \emph{congruence} of a first-order structure~$A$ as an equivalence relation on (the universe of)~$A$, augmented by a family of subsets of finite powers of~$A$ indexed by the set of all relation symbols of~$\scL$, satisfying certain compatibility conditions. (In particular, if there are no relation symbols, then a congruence is an equivalence relation.) This definition is equivalent to the one introduced in~\cite{Gorb}. Congruences of a first-order structure~$A$ are in one-to-one correspondence with surjective homomorphisms with domain~$A$ up to isomorphism. A most important class of examples, extensively considered in~\cite{Gorb}, is given by \emph{graphs}.

Unlike varieties, quasivarieties may not be closed under homomorphic images. Thus the relevant concept of congruence, for a member~$A$ of a quasivariety~$\cV$, is often modified by considering only the \emph{$\cV$-congruences} (or \emph{congruences relative to~$\cV$}): by definition, a congruence~$\theta$ of~$A$ is a $\cV$-congruence if the quotient~$A/\theta$ is a member of~$\cV$. The set~$\ConV A$ of all $\cV$-congruences of~$A$, partially ordered by inclusion, is still an algebraic lattice, and we denote by~$\ConcV A$ its \jzs\ of compact elements. Furthermore, we define $\Concr\cV$ (where the letter ``r'' stands for ``relative'') as the class of all isomorphic copies of $\ConcV A$ for $A\in\cV$.

The \emph{relative critical point} $\critr{\cV}{\cW}$ between quasivarieties~$\cV$ and~$\cW$ of first-order structures is defined as the least cardinality of a member of the difference $(\Concr\cV)-(\Concr\cW)$, if $\Concr\cV\not\subseteq\Concr\cW$; and~$\infty$ otherwise. Our main result, Theorem~\ref{T:poscritVar}, extends \emph{mutatis mutandis} to quasivarieties of first-order structures and relative congruence lattices. Aside from a few additional arguments dealing with the relation symbols, very similar to those dealing with equality for varieties of algebras, the proofs are virtually the same. An important point is the finiteness assumption on the sets of relation symbols in the languages of both~$\cV$ and~$\cW$, which is essentially required in order to ensure that~$\Con A$ is finite whenever (the universe of)~$A$ is finite. We thus improve the estimate of the Dichotomy Theorem stated in \cite[Section~4-9]{GiWe1}, namely $\critr{\cV}{\cW}<\aleph_\omega$, to $\critr{\cV}{\cW}\le\aleph_2$, moreover under a slightly weaker assumption.

\begin{theorem}\label{T:poscritQVar}
Let~$\cV$ and~$\cW$ be locally finite quasivarieties of first-order structures, in first-order languages with only finitely many relation symbols. Assume that for each finite $A\in\cV$ there are, up to isomorphism, only finitely many $B\in\cW$ such that $\ConcV A\cong\ConcW B$, and every such~$B$ is finite. Then either $\critr{\cV}{\cW}\le\aleph_2$ or $\Concr\cV\subseteq\Concr\cW$.
\end{theorem}

It is still not known whether the assumption, stating that for any finite $A\in\cV$ there are only finitely many $B\in\cB$ such that $\ConcV A\cong\ConcW B$, can be dispensed with. Due to the example in~\cite[Exercise~14.9(4)]{HM}, attributed there to C. Shallon, this assumption does not hold as a rule: there is a finitely generated variety of algebras with a proper class of simple members. We also do not know whether the local finiteness assumption, on both quasivarieties~$\cV$ and~$\cW$, can be dispensed with.

\section{Acknowledgment}
I thank Friedrich Wehrung, for many constructive remarks and the help to write this paper.

\end{document}